\documentclass[twoside,11pt]{article}
\usepackage{amssymb,amsmath,amsthm,amsfonts, epsfig}           
\usepackage{mathtools} 
\usepackage{hyperref}
\usepackage{eucal}
\usepackage{a4wide}
\usepackage[svgnames]{xcolor}
\usepackage{mathptmx} 
\usepackage{graphicx}
\usepackage{tikz-cd}
 \usepackage{fancyhdr}
 \usepackage{MnSymbol}
 \usepackage{subcaption}

\title{Invariant measures as obstructions to  attractors in dynamical systems and  their role in nonholonomic mechanics  }

\author{  L.~C.~Garc\'ia-Naranjo\footnote{Dipartimento di Matematica ``Tullio Levi-Civita", Universit\`a di Padova, Via Trieste 63, 35121 Padova, 
Italy. luis.garcianaranjo@math.unipd.it. ORCID: 0000-0002-3589-6068.}\, , \, R. Ortega\footnote{Departmento de Matem\'atica Aplicada, Facultad de Ciencias, Universidad de Granada, 18071 Granada, Spain. R.O.: rortega@ugr.es. ORCID: 0000-0002-7968-1739. A.J.U: ajurena@ugr.es, ORCID: 0000-0002-5384-7251.} \; and A.~J.~Ure\~na$^{\dagger}$ }

\date{\today}

\numberwithin{equation}{section}
\numberwithin{table}{section}
\numberwithin{figure}{section}

\hypersetup{colorlinks=true,linkcolor=violet,citecolor=purple}

\newtheorem{theorem}{Theorem}[section]
\newtheorem{lemma}[theorem]{Lemma}
\newtheorem{proposition}[theorem]{Proposition}

{\theoremstyle{definition}

\newtheorem{remark}[theorem]{Remark}
\newtheorem*{remarks*}{Remarks}

}

\newtheorem{innercustomgeneric}{\customgenericname}
\providecommand{\customgenericname}{}
\newcommand{\newcustomtheorem}[2]{%
  \newenvironment{#1}[1]
  {%
   \renewcommand\customgenericname{#2}%
   \renewcommand\theinnercustomgeneric{##1}%
   \innercustomgeneric
  }
  {\endinnercustomgeneric}
}

\newcustomtheorem{customhyp}{}

\def\headcolour{\color{Grey}}
\def\restr#1{\,\vrule height1.2ex width.4pt
  depth0.8ex\lower0.4ex\hbox{\scriptsize $\,#1$}}

\newcommand{\R}{\mathbb{R}}

\newcommand{\I}{\mathbb{I}}

\newcommand{\so}{\mathfrak{so}}

\newcommand{\A}{\mathcal{A}}

\begin{document}

\maketitle

\abstract{We present some rigorous results on the absence of a wide class of invariant measures for dynamical systems possessing attractors. 
We then consider a generalization of the classical   nonholonomic Suslov problem  which shows how previous  investigations of existence of 
invariant measures for nonholonomic
systems should necessarily  be extended beyond the class of measures with strictly positive $C^1$ densities 
if one wishes  to determine dynamical obstructions to the presence of attractors. }

\vspace{0.5cm}

\section{Introduction}
\label{S:intro}

 It is a well-established idea that the existence of an invariant measure is an  important feature of a dynamical
system. However, the  type of information  that can be extracted from such
an invariant measure depends on its properties. For instance, a measure concentrated at a point is invariant for a system if and only if the point is an equilibrium. It is clear that no additional information about the system can be extracted from the invariance of such measure. This paper focuses on the invariance of a specific class of measures having an almost everywhere positive density that results in the obstruction to the
existence of attractors in the phase space.

Our motivation comes from nonholonomic mechanics. It is well known that this type of mechanical systems
does not possess a symplectic Hamiltonian structure and, therefore, the existence of a smooth invariant
volume form on the phase space is not guaranteed. Even though nonholonomic systems are 
energy preserving,
the absence of an invariant measure  may lead to surprising phenomena like the reversal of the rattleback (see e.g. \cite{BKM06,Betal2012,AKoch22}
and the references therein) which cannot occur in symplectic Hamiltonian systems.
 
Since  Blackall's early work \cite{Blackall}, a great number of references   have  analyzed the conditions
of existence of smooth invariant measures for nonholonomic systems
e.g.  \cite{Kozlov87, Kozlov88,Stanchenko89,Jov98,CaCoLeMa02,ZeBloch03,BoMaBi13, FedGNMa2015,ClarkBloch23}.
A common point in these analyses is that the density  of the invariant measure
(with respect to a certain volume form) is assumed to be 
everywhere  positive on the phase space. This assumption is  used to  work with the logarithm 
of the density function. 

The point of this paper is to indicate that the  condition that the density of the measure
is everywhere positive is too strong if one is interested in determining obstructions to the 
existence of attractors in the phase space. For this reason, we relax this condition
and  allow the densities of  measures in the class $\A$ that we consider  to vanish on a set of measure zero.
To illustrate the potential relevance of these measures in nonholonomic systems,  we provide a concrete  example,
consisting of a generalization of the classical Suslov problem,  which for certain parameter values 
possesses  an invariant  measure of type $\A$ whose density 
necessarily vanishes. In particular, the existence of this invariant measure cannot be detected with the techniques of the references indicated
above.

The relevance of  integral invariants which are not necessarily positive everywhere has been indicated
before by Kozlov \cite[Section 5]{Kozlov16}. The present paper complements
his discussion in the context of obstructions to the existence of attractors. Moreover,   we clarify their
role in nonholonomic mechanics with our example.

The  paper is organized as  follows. We first introduce  invariant measures of class  $\A$ in Section \ref{S:measures}
and prove that they are obstructions for the existence of attractors in Proposition \ref{P:attractor-measure}. In Section 
\ref{S:prob} we study  our example. After deriving
the reduced equations of motion and their main properties, we discuss existence
of invariant measures in Section \ref{ss:measureExample}.  We finally give some concluding 
remarks in Section \ref{S:conclusions}.

\section{  Invariant measures and attractors in dynamical systems} 
\label{S:measures}

In the Euclidean space $\R^N$ we consider measures $\mu$ of the type
\begin{equation}
\label{eq:in} \mu = M(x_1,\dots, x_N) \, dx_1\cdots dx_N,
\end{equation}
where the density function $M:\R^N\to \R$ is  measurable, locally integrable and strictly positive almost everywhere. This class of measures will be
denoted by $\A$.

 Assume now that $X:\R^N\to \R^N$ is a $C^1$ vector field such that the system of equations
\begin{equation}
\label{eq:de} \dot x = X(x),
\end{equation}
defines a complete flow $\{\phi_t\}_{t\in \R}$ in $\R^N$. We say that the measure $\mu\in \A$ is invariant under the flow
if 
\begin{equation*}
\mu(\phi_t(A))=\mu(A),
\end{equation*}
for each $t\in \R$ and each Lebesgue measurable set $A\subset \R^N$. This definition makes sense since
diffeomorphisms preserve the  $\sigma$-algebra of
Lebesgue measurable sets.

The book \cite{NemStepanov} contains an excellent introduction to the theory of invariant measures. Following the original terminology of Poincar\'e, and allowing a certain ambiguity, these measures are sometimes
called  {\em integral invariants}.
  In particular, it is proved in 
 \cite[p.429]{NemStepanov} that the partial differential equation
 \begin{equation}
\label{eq:pde} \sum_{i=1}^N \frac{\partial }{\partial  x_i} \left ( MX_i \right ) = 0,
\end{equation}
can be employed to characterize invariant measures with $M\in C^1(\R^N)$.

In contrast with some previous literature mentioned in the introduction, we are allowing  sets of measure zero where the density can vanish. This is useful to enlarge
the class of systems \eqref{eq:de} with invariant measure. As a simple example, consider the linear system in $\R^2$,
\begin{equation}
\label{eq:2Dexample}
\dot x_1= -x_1, \qquad \dot x_2=2x_2.
\end{equation}
The measure defined by \eqref{eq:in} with $M(x_1,x_2)=|x_1|^5x_2^2$ is invariant. In fact $M$ is in $C^1(\R^2)$ and it is easy to check that
it satisfies \eqref{eq:pde}. Since $M$ only vanishes on the coordinate axes we note that the measure belongs to the class $\A$. On the other hand,
this linear system cannot admit an invariant measure $\mu$ of the form given in \eqref{eq:in} with $M\in C^1(\R^2)$ and $M(x_1,x_2)>0$ for every $(x_1,x_2)\in \R^2$. The equation \eqref{eq:pde}
can be expressed as 
\begin{equation*}
-x_1\frac{\partial  M}{\partial  x_1}+2 x_2\frac{\partial M}{\partial  x_2}+M=0,
\end{equation*}
and it is clear that any solution must vanish at the origin.

Assume now that $K\subset \R^N$ is a non-empty compact set that is invariant under the flow; that is,
\begin{equation*}
\phi_t(K)=K \qquad \mbox{for each $t\in \R$}.
\end{equation*}
We say that  $K$ is an {\em attractor} if there exists an open set $W\subset \R^N$ with $K\subset W$
and such that 
\begin{equation*}
\mbox{dist}(\phi_t(x),K)\to 0 \quad \mbox{as} \quad t\to +\infty
\end{equation*}
for each $x\in W$.

It is intuitively clear that class $\A$ invariant measures cannot exist in the presence of an attractor. The following proposition
 formalizes  this idea. In its proof, we  use the following  two properties of class $\A$ measures which 
 follow directly from their definition: 1) the measure of any non-empty open set is positive; and 2)  the measure of any bounded set is finite.

\begin{proposition} 
\label{P:attractor-measure}
Assume that the flow associated to the system \eqref{eq:de} has an attractor. Then there are no invariant measures in the class $\A$.
\end{proposition}
\begin{proof}
By contradiction assume that $\mu\in \A$ is invariant. Let $W$ be an open and bounded set with $K\subset W$ and attracted by $K$. 
We fix a small number $r>0$ such that $K_r\subset W$, where
\begin{equation}
\label{eq:auxsetKr}
K_r=\{ x\in \R^N \, :\, \mbox{dist}(x,K)\leq r \}.
\end{equation}
Since $W\setminus K_r$ is a non-empty open set, we know that it has positive measure. Let us take $\delta >0$ with 
\begin{equation}
\label{eq:Ant}
\delta < \mu(W\setminus K_r)= \mu(W)- \mu( K_r).
\end{equation}
Now consider the sequence of measurable functions
\begin{equation*}
\chi_{K_r} \circ\phi_{t_n} :W\to \R,
\end{equation*}
where $\chi_{K_r}$ is the characteristic function of $K_r$ and $\{t_n\}$ is a sequence with $t_n\to +\infty$. Since $K$ attracts $W$, we obtain
\begin{equation*}
\chi_{K_r} \circ\phi_{t_n}(x)\to 1 \qquad \mbox{as} \quad n\to +\infty,
\end{equation*}
for each $x\in W$. In principle this convergence must be understood in a pointwise sense but we can apply Egoroff's Theorem (see e.g. \cite{EvansGariepy})
to obtain uniform convergence in a suitable set. More precisely, we can find a measurable set $W_*\subset W$ such that 
$\mu(W\setminus W_*)<\delta$ and the above convergence is uniform in $x\in W_*$. Then, for $n$ large enough, we know that
$\phi_{t_n}(W_*)$ is contained in $K_r$. In particular, $\mu(\phi_{t_n}(W_*))\leq \mu(K_r)$. 
 Using the invariance of the measure we find,
\begin{equation*}
\begin{split}
\mu(W)&=\mu(W_*)+\mu(W\setminus W_*)\\ &= \mu(\phi_{t_n}(W_*))+ \mu(W\setminus W_*) \\
&< \mu(K_r)+\delta,
\end{split}
\end{equation*}
 which is incompatible with \eqref{eq:Ant} since $\mu(K_r)$ and $\mu(W)$ are finite.
\end{proof}

%

\begin{remark}
\label{rmk:measurezeroW}
The hypothesis that the region of attraction $W$ contains the compact invariant set  $K$ may be weakened if $K$ has
zero measure. Specifically, the conclusion of the proposition remains valid  under the following condition:
{\em There exists a non-empty open set $W$ with $\phi_t(x)\to K$ if $t\to +\infty$ for each $x\in W$.} 
Indeed, in this case one can obviously find $\delta, r>0$ such that \eqref{eq:Ant} holds with $K_r$ defined by  \eqref{eq:auxsetKr}, and the rest of
the proof applies unchanged.
This can be used to conclude non-existence of class $\A$ invariant 
measures in the presence of partial attractors. As an example consider $N=1$ and $\dot x_1=\sin ^2(x_1)$, $K=\{0\}$.
\end{remark}

For the application we have in mind we need to replace $\R^N$ by a finite dimensional manifold. This can be done without too much effort. 
Assume that $(\mathcal{M},g)$ is a Riemannian manifold of dimension $n$ and let $\mu_g$ be the associated measure. In local
coordinates $(x_1,\dots, x_N)$ the measure $\mu_g$ is defined by the formula
\begin{equation*}
\sqrt{\det g(x_1,\dots, x_N)}\, dx_1\cdots dx_N.
\end{equation*}
Then we will consider the class $\A(\mathcal{M})$ of admissible measures
\begin{equation*}
\mu = M \mu_g,
\end{equation*}
where $M:\mathcal{M}\to \R$ is a locally integrable function (with respect to $\mu_g$) that is strictly positive almost everywhere in $\mathcal{M}$.
The reader interested in measure theory will have no problems to prove that the class $\A(\mathcal{M})$ is independent of the metric $g$, however,
this fact will not be used in this paper. 
%

Given a vector field $X$ on $\mathcal{M}$, we can consider the system \eqref{eq:de} on the manifold and it will be assumed that the 
associated flow $\{\phi_t\}_{t\in \R}$ is complete on $\mathcal{M}$. The notion of attractor is defined as before
and the proof of  Proposition \ref{P:attractor-measure} is easily adapted\footnote{the adaptation requires the replacement of 
bounded open sets in $\R^N$ by relatively compact open sets in the manifold $\mathcal{M}$.} to this setting.

\section{Nonholonomic example: a generalization of the Suslov problem}
\label{S:prob}

In this section we introduce our example which is a generalization of the classical nonholonomic Suslov problem for the dynamics of a rigid body.
Our generalization  consists in the incorporation of an internal  rotor,
which adds a degree of freedom to the system, leading to  a more delicate dependence of the dynamics on the system parameters.
Our main point is to discuss existence of invariant measures which we do  in subsection \ref{ss:measureExample}.

\subsection{Description of the system}
Consider a rigid body, which we refer to as the {\em carrier}, with a rotor in its inside.  For simplicity,  we suppose that
the center of mass of the carrier and the   rotor coincide and  is fixed throughout the motion. Also for simplicity,
we assume  that the axis of symmetry of the rotor is aligned with  the smallest axis of inertia of the carrier that we label as the third one.
The angular velocity of the rotor  is then $\Omega_r= \Omega + \dot \theta E_3$ where $\Omega$ is the angular velocity of the carrier,  $E_3=(0,0,1)$, and $\theta$ is the rotation  angle of the rotor about its axis (see Figure \ref{F:scheme}). All of the vectors given above, and all those appearing hereafter,  are written with respect to a frame attached to the carrier which is
centred at its center of mass and is aligned with its principal axes of inertia. 

\begin{figure}[h!]
\centering
\includegraphics[width=4cm]{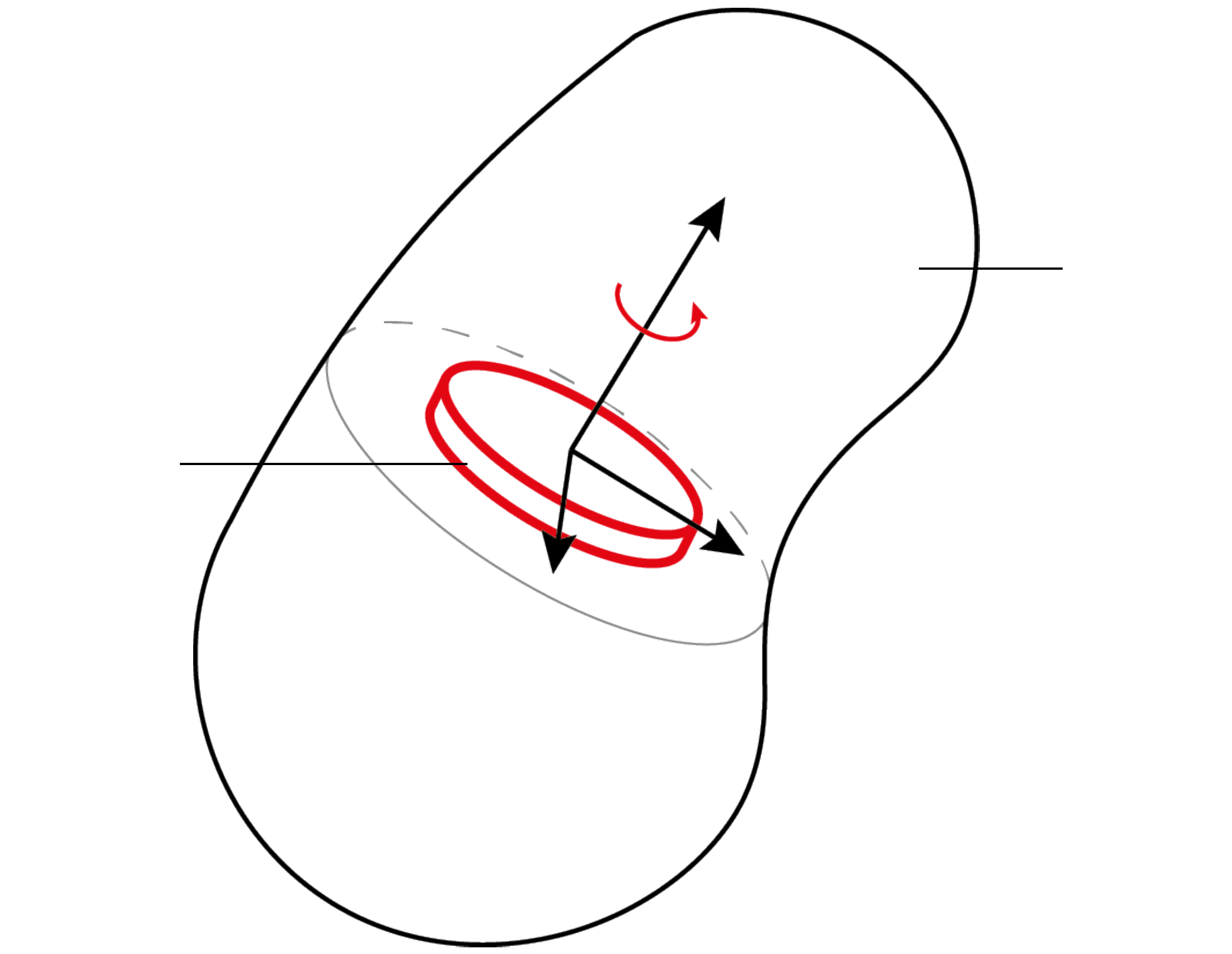} 
 \put (-14,61) {\small{carrier}}
    \put (-50,50) {\small{$\theta$}}
      \put (-44,65) {\small{$E_3$}}
          \put (-68,28) {\small{$E_1$}}
                 \put (-120,43) {rotor}
                  \put (-45,33) {\small{$E_2$}}
\caption{{\small Schematic representation of the carrier-rotor system.
}}
\label{F:scheme}
\end{figure}
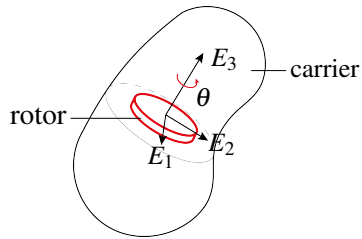

The configuration space of the problem is the direct product Lie group $Q=SO(3)\times S^1$ where elements of $SO(3)$ specify the orientation
of the carrier relative to an inertial frame and the angle $\theta \in S^1$ gives the orientation of the rotor relative to the carrier as explained above. 
The kinetic energy  Lagrangian is
\begin{equation*}
L=\frac{1}{2} \langle \I_0 \Omega, \Omega \rangle +\frac{1}{2} \langle \I_r ( \Omega + \dot \theta E_3),  \Omega + \dot \theta E_3 \rangle ,
\end{equation*}
where $ \I_0=\mbox{diag}(I_1,I_2,I_3)$ is the inertia tensor of the carrier and $ \I_r=\mbox{diag}(K_1,K_1,K_3)$ is the inertia tensor of the rotor
(we assume that $K_3>0$ and $K_1\geq 0$). We assume that  the principal moments of inertia of the carrier satisfy
\begin{equation}
\label{eq:mom-inertia}
0<I_3<I_2<I_1.
\end{equation}

We consider the evolution of the above Lagrangian system subject to a Suslov-type nonholonomic constraint imposed to the
angular velocity $\Omega_r$ of the rotor. Specifically, we require that
\begin{equation}
\label{eq:const}
\langle a ,   \Omega_r  \rangle=\langle a ,   \Omega + \dot \theta E_3 \rangle=0,
\end{equation}
for a fixed non-zero vector $a\in \R^3$. Physically this means that the component of the (total) angular velocity of the rotor about an axis which is {\em fixed in the carrier} vanishes. This 
problem may be understood as a
generalization of  the classical Suslov problem \cite{Suslov}.\footnote{We recall that the Suslov problem considers the motion of a rigid body, with no rotor, subject to 
 the constraint $\langle a , \Omega \rangle =0$.} In the following, we will refer to $a$ as the {\em vector of forbidden rotations} and denote its components by 
 $a:=(a_1,a_2,a_3)$.

The phase space of the system, consisting of all configurations and admissible velocities is a 7-dimensional manifold  diffeomorphic 
to $SO(3)\times S^1\times \R^3$.

\subsection{Equations of motion}

Since both the Lagrangian and the constraints are written in solely in terms of $\Omega$ and $\dot \theta$, the constraint distribution and the 
Lagrangian are invariant under left multiplication on the direct product Lie group $SO(3)\times S^1$. Nonholonomic systems of this type, whose configuration space
is a Lie group and possess the invariance properties mentioned above, are 
called {\em LL-systems} (see e.g. \cite{Jov01, FedZen}). Following  \cite{FedKoz}, the reduced equations of motion are termed {\em Euler-Poincar\'e-Suslov equations}. In our
case these   are given by
\begin{equation}
\label{eq:EPS1}
\begin{split}
\frac{d}{dt} \left ( \frac{\partial L}{\partial \Omega} \right ) &= \frac{\partial L}{\partial \Omega} \times \Omega + \zeta a, \\
\frac{d}{dt} \left ( \frac{\partial L}{\partial \dot \theta} \right ) &=   \zeta a_3,
\end{split}
\end{equation}
for a multiplier  $\zeta$ that will be determined below.

 If $a_3=0$, i.e. if the vector of forbidden rotations is perpendicular to the axis of the rotor,
then the constraint does not involve $\dot \theta$, the second equation in \eqref{eq:EPS1} becomes a conservation law and  the system may be shown to be equivalent to the Suslov problem 
with gyroscope studied before in \cite{MacPrz}.  In fact, the resulting system further simplifies to the classical Suslov problem since the axis of the
gyroscope is perpendicular to the vector $a$ of forbidden rotations (see \cite{MacPrz}).

For the rest of the paper we  will focus on the the case $a_3\neq 0$, and assume, without loss of generality, that
\begin{equation*}
a_3=1.
\end{equation*}

We now proceed to determine the multiplier $\zeta$ in terms of $\dot \Omega$. On the one hand, differentiation of the constraint \eqref{eq:const} gives 
\begin{equation}
\label{eq:lambda-aux1}
 \ddot \theta +\langle  a , \dot \Omega \rangle =0.
\end{equation}
On the other hand, considering that $ \frac{\partial L}{\partial \dot \theta}=K_3( \dot \theta +\Omega_3)$, where $\Omega:=(\Omega_1,\Omega_2,\Omega_3)$, 
the second equation in \eqref{eq:EPS1} yields
\begin{equation}
\label{eq:lambda-aux2}
\zeta=K_3(\ddot \theta +\dot  \Omega_3). 
\end{equation}
Combining  \eqref{eq:lambda-aux1} and  \eqref{eq:lambda-aux2} leads to our desired expression for $\zeta$:
\begin{equation*}
\label{eq:lambda}
\zeta= K_3 (\dot \Omega_3 - \langle  a , \dot \Omega \rangle ).
\end{equation*}

We now use the above expression for $\zeta$ and the constraint  \eqref{eq:const} to rewrite the first equation in   \eqref{eq:EPS1} as an autonomous 
system in 3 unknowns, $\Omega_1, \Omega_2, \Omega_3$, describing the evolution of the reduced system. 
Using that  $\frac{\partial L}{\partial \Omega} =(\I_0+\I_r)\Omega + K_3 \dot \theta E_3$, we get
\begin{equation}
\label{eq:motion-aux}
\frac{d}{dt} \left ((\I_0+\I_r)\Omega - K_3\langle a, \Omega \rangle E_3  \right ) = \left ((\I_0+\I_r)\Omega - K_3\langle a, \Omega \rangle E_3  \right )\times \Omega
+ K_3 \left (   \dot \Omega_3 - \langle  a , \dot \Omega \rangle \right ) a.
\end{equation}
The above equation may be written in a more appealing form by introducing the matrices\footnote{For column vectors $b,c\in \R^3$ we denote by $b\otimes c$ the $3\times 3$ matrix 
given by $bc^T$ where  ${}^T$ denotes transposition. It is clear that $b\otimes c$ has rank 1 if both $b$ and $c$ are nonzero. }
\begin{equation*}
\label{eq:mat-def}
\begin{split}
&\tilde{ \I}:=\I_0+\I_r - K_3 E_3\otimes E_3 = \mbox{diag}(I_1+K_1, I_2+K_1, I_3),\\
& \mathbb{K}_a:= \tilde{ \I}+K_3 \left (E_3-a \right )\otimes \left (E_3-a \right ), \\
& \mathbb{B}_a:=\tilde{ \I}+K_3 E_3 \otimes  \left (E_3-a \right ).
\end{split}
\end{equation*}
A calculation shows that \eqref{eq:motion-aux} is equivalently written as the following  system for the evolution of $\Omega\in \R^3$:
\begin{equation}
\label{eq:motionB}
\frac{d}{dt} \left ( \mathbb{K}_a \Omega   \right ) =\left   ( \mathbb{B}_a \Omega  \right )\times \Omega.
\end{equation}
For future reference we note that the explicit form of the matrices $ \mathbb{K}_a$ and  $\mathbb{B}_a$ is: 
\begin{equation*}
 \mathbb{K}_a=\left ( \begin{matrix} \lambda_1 +a_1^2K_3 & a_1a_2K_3 & 0 \\ a_1a_2K_3 &  \lambda_2 +a_2^2K_3 & 0
  \\ 0 & 0 & \lambda_3 \end{matrix} \right ), \qquad  \mathbb{B}_a=\left ( \begin{matrix} \lambda_1 &0 & 0 \\ 0 &  \lambda_2  & 0
  \\ -a_1K_3 & -a_2 K_3& \lambda_3 \end{matrix} \right ),
\end{equation*}
with
\begin{equation*}
\lambda_1=I_1+K_1, \qquad \lambda_2=I_2+K_1, \qquad \lambda_3=I_3.
\end{equation*}
We also note that our assumption \eqref{eq:mom-inertia} implies 
\begin{equation}
\label{eq:mom-inertia-lambda}
0<\lambda_3<\lambda_2<\lambda_1.
\end{equation}

The full (unreduced) equations of motion of the system on the 7-dimensional phase space  consist of the reduced system \eqref{eq:motionB} 
complemented with the reconstruction equations for the attitude matrix $g\in SO(3)$ of the
carrier and the angle $\theta$. These  are given by the kinematical relations 
\begin{equation*}
\label{eq:reconst}
\dot g = g \hat \Omega, \qquad \dot \theta = -\langle a , \Omega\rangle. 
\end{equation*}
As is usual, $ \hat \Omega$ denotes the unique $3\times 3$ skew-symmetric matrix such that $ \hat \Omega v=  \Omega \times v$ for all $v\in \R^3$.

\subsection{Conservation of energy}

It is clear from \eqref{eq:motionB}  that the energy 
\begin{equation}
E(\Omega):=\frac{1}{2}\langle  \mathbb{K}_a \Omega , \Omega\rangle,
\end{equation}
 is a first integral. We observe that $  \mathbb{K}_a$ is 
 symmetric and positive definite and, therefore, the level sets
 of $E$ for the reduced system are ellipsoids that we denote by 
 $$\mathcal{E}_\eta:=\left \{\Omega\in \R^3 \, : \, E(\Omega)=\eta>0 \right \}.$$
 
 Since the equations \eqref{eq:motionB} are homogeneous quadratic in $\Omega$, it is easy to show that
if $t\mapsto \Omega(t)$ is a solution of  \eqref{eq:motionB}, then so is $t\mapsto c\Omega(ct)$ for any $c\in \R$. In particular, 
since $E$ is also homogeneous quadratic in $\Omega$, it follows that 
 the dynamics on the different constant energy ellipsoids $\mathcal{E}_\eta$,
$\eta\neq 0$, 
is conjugated by a constant time rescaling. 

\subsection{Equilibria}

The nonzero equilibrium points of \eqref{eq:motionB} correspond to steady rotations of the carrier with constant angular velocity. We now classify these solutions and determine their stability on 
each constant energy ellipsoid  $\mathcal{E}_\eta$.

It is clear from \eqref{eq:motionB} that these equilibrium points are precisely the eigenvectors of  $ \mathbb{B}_a$. By our assumption 
\eqref{eq:mom-inertia}, the matrix
 $ \mathbb{B}_a$ has distinct (real) eigenvalues   $\lambda_3<\lambda_2<\lambda_1$, so we conclude that 
 the set of equilibria of \eqref{eq:motionB} is composed by three 
 lines $V_i$, $i=1,2,3$, passing through the origin
\begin{equation*}
\mbox{set of equilibria}= \bigcup_{i=1}^3V_i, \qquad V_i:=\ker (\mathbb{B}_a -\lambda_i I).
\end{equation*}
Each constant energy ellipsoid  $\mathcal{E}_\eta$, with $\eta>0$, intersects the line $V_i$ at two opposite points $\pm v_i$, so there is a total of 6 equilibrium 
points on $\mathcal{E}_\eta$ organized in 3  pairs $\pm v_i$, $i=1,2, 3$ (see Figure \ref{F:eqptsE1}). When analyzing their stability ahead it is useful to keep in mind that the  
dynamics near $v_i$ coincides with that near $-v_i$  if the sense of time is reversed.
This is a consequence of the reversibility of the dynamics of  \eqref{eq:motionB}, i.e. 
if $t\mapsto \Omega(t)$ is a solution so is $t\mapsto -\Omega(-t)$.

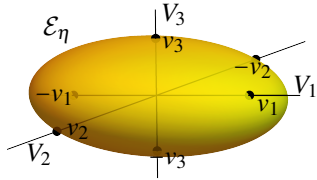
\begin{figure}[h!]
\centering
\includegraphics[width=4cm]{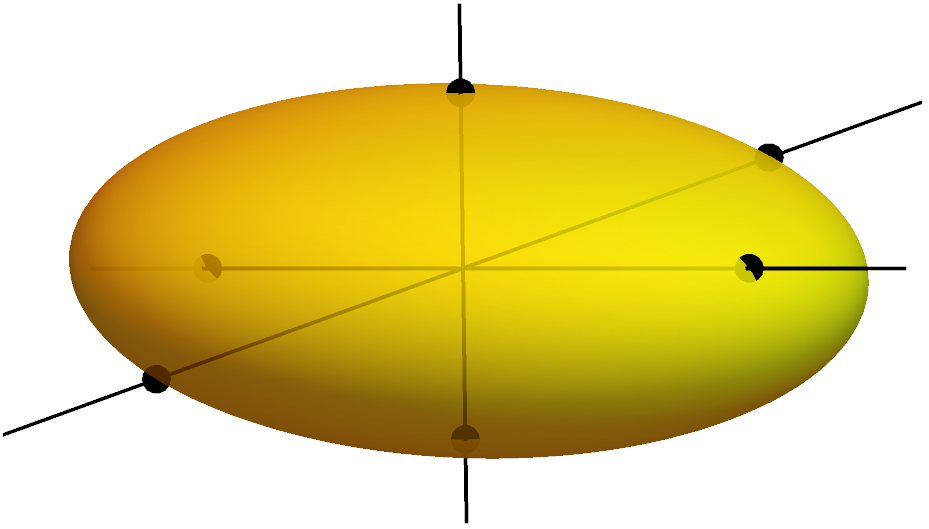} 
 \put (-100,55) {\small{$\mathcal{E}_\eta$}}
 \put (-5,33) {\small{$V_1$}}
  \put (-19,26) {\small{$v_1$}}
    \put (-55,50) {\small{$v_3$}}
      \put (-55,60) {\small{$V_3$}}
       \put (-60,6) {\small{$-v_3$}}
          \put (-91,17) {\small{$v_2$}}
               \put (-106,7) {\small{$V_2$}}
                 \put (-103,30) {\small{$-v_1$}}
                  \put (-28,40) {\small{$-v_2$}}
\caption{{\small The lines $V_i$ and the 6 equilibrium points $\pm v_i$ on a constant energy ellipsoid $\mathcal{E}_\eta$.
}}
\label{F:eqptsE1}
\end{figure}

\subsubsection{Stability analysis}
\label{ss:stability}
The linearization of  \eqref{eq:motionB}  at the equilibrium point $0\neq v_i\in V_i$, i.e. corresponding to the eigenvalue $\lambda_i$, is given by
\begin{equation}
\label{eq:var}
\mathbb{K}_a \dot \delta = G_{v_i} \delta,
\end{equation}
 where $G_{v_i}$ is the endomorphism of $\R^3$ defined by  
\begin{equation*}
 G_{v_i} \delta = [ (\mathbb{B}_a -\lambda_i I) \delta ]\times v_i.
\end{equation*}
Since $G_{v_i}v_i=0$, the linear system \eqref{eq:var} has the line of equilibria $V_i=\langle v_i\rangle$. A more geometric 
insight of the linearized system may be obtained by recalling that the constant energy ellipsoids  are invariant under the nonlinear flow 
 \eqref{eq:motionB}. As a consequence, the tangent plane $T_{v_i}\mathcal{E}_\eta$ to  the ellipsoid  $\mathcal{E}_\eta$, $\eta=E(v_i)$,  is invariant by the linearized flow \eqref{eq:var}.
In other words, we can split $\R^3$ in the form
\begin{equation*}
\R^3=V_i\oplus W_i \quad \mbox{with} \quad W_i:=T_{v_i}\mathcal{E}_\eta=(\nabla_\Omega E(v_i))^\perp=(\mathbb{K}_aV_i)^\perp,
\end{equation*}
and
\begin{equation*}
\mathbb{K}_a^{-1}G_{v_i}(W_i)\subseteq W_i, \qquad \mathbb{K}_a^{-1}G_{v_i}=0 \; \mbox{on $V_i$.}
\end{equation*}
As a consequence, the characteristic polynomial  $ \tilde p(z):=\det(zI - \mathbb{K}_a^{-1}G_{v_i})$ can be factorized as 
\begin{equation*}
 \tilde p(z)=z(z^2+ \tilde \alpha z+ \tilde \beta),
\end{equation*}
where the quadratic factor corresponds to the characteristic polynomial of the restricted endomorphism $\left . \mathbb{K}_a^{-1}G_{v_i} \right |_{W_i} :W_i\to W_i$.
The classical classification for linear equilibria in $\R^2$ in terms of the signs of the trace $- \tilde\alpha$ and the determinant $\tilde \beta$ may then be applied
to determine the nature of the equilibrium point of the restriction of the flow to $\mathcal{E}_\eta$. These observations
lead to the following.

\begin{lemma}
\label{l:AJ}
Let $0\neq v_i\in V_i$ be an equilibrium point of  \eqref{eq:motionB} with $\eta= E(v_i)>0$.  
 Define $ p(z):=\det(z\mathbb{K}_a-G_{v_i})$. Then $p(z)=\det( \mathbb{K}_a)z^3 + \alpha z^2+ \beta z$ with $\alpha, \beta \in \R$   and:
\begin{enumerate}
\item if $\beta<0$ then $\pm v_i$ are saddle points of the restricted flow to $\mathcal{E}_\eta$;
\item if $\beta>0$  then the equilibrium points $\pm v_i$ of the restricted flow to $\mathcal{E}_\eta$ satisfy:
\begin{enumerate}
\item they are both   linear centers if $\alpha=0$;
\item $v_i$ is a source and $-v_i$ a sink if $\alpha<0$;
\item $v_i$ is a sink and $-v_i$ a source if $\alpha>0$.
\end{enumerate}
\end{enumerate}
\end{lemma}
\begin{proof}
We have $p(z)=\det(\mathbb{K}_a) \tilde p(z)$ and hence $p(z)$ may be written as indicated with $\alpha= \det(\mathbb{K}_a) \tilde \alpha$ and $\beta= \det(\mathbb{K}_a) \tilde \beta$.
Considering that $\det(\mathbb{K}_a)>0$, we conclude  that the signs of $-\alpha$ and
$\beta$ respectively coincide with the signs of the trace and determinant of $\left . \mathbb{K}_a^{-1}G_{v_i} \right |_{W_i} :W_i\to W_i$. 
The specific nature of the equilibrium point $v_i$ described in items $(i), (ii)$, 
follows, as was mentioned in the text above, 
 from the standard classification of equilibria of linear systems in $\R^2$ in terms of the signs of the trace and the determinant of the linearization matrix.
The conclusions about the behavior at $-v_i$ follow from the reversibility of the dynamics. 
\end{proof}

Below we use this lemma to determine the linear stability of the equilibrium points of the restriction of the flow to 
a constant energy ellipsoid $\mathcal{E}_\eta$. For this we note that the  matrix form of $G_{v_i}$ is 
\begin{equation*}
 G_{v_i}=\left ( g_1 \right | \left . g_2 \right | \left . g_3 \right ), \qquad
g_j=(b_j-\lambda_i e_j)\times v_i, \;\;\; j=1,2,3,
\end{equation*}
where $b_j$, $j=1,2,3$ are the columns of $\mathbb{B}_a$. Using the  explicit expressions
for the matrices $G_{v_i}$ and $\mathbb{K}_a$, the coefficients $\alpha$ and $\beta$ in the lemma can be
 conveniently calculated with the help of  a symbolic algebra
software and their sign can be determined considering \eqref{eq:mom-inertia-lambda}. Finally, note that the stability properties determined at a  specific equilibrium point $v_i\in V_i$ may be extended to the whole line  $V_i$ (excluding the origin) 
since,  as mentioned above, the dynamics on the different ellipsoids $\mathcal{E}_\eta$ is conjugated by a constant time rescaling.

\paragraph{Linearization at $\pm v_1\in V_1$.} Taking  $v_1$ as the column vector $(\lambda_3-\lambda_1, 0, a_1K_3)$ one finds
\begin{equation*}
\begin{split}
\beta&=(\lambda_1- \lambda_2) (\lambda_1-\lambda_3)
   \left(\left(a_1^2K_3+ \lambda_1\right) ( \lambda_1 -\lambda_3)^2+a_1^2K_3^2 \lambda_3\right)>0, \\
   \alpha&= -a_2K_3 \lambda_3\left (a_1^2K_3(\lambda_1-\lambda_2)+\lambda_1(\lambda_1-\lambda_3) \right ).
   \end{split}
\end{equation*}
Therefore, using Lemma \ref{l:AJ}, we conclude that $\pm v_1$ are a source and a sink on the ellipsoid $\mathcal{E}_\eta$
if $a_2\neq 0$ and instead two linear centers if $a_2=0$.

\paragraph{Linearization at $\pm v_2\in V_2$.} For the column vector  $v_2=(0,\lambda_3-\lambda_2, a_2K_3)$ one computes
\begin{equation*}
\beta=-(\lambda_1- \lambda_2) (\lambda_2-\lambda_3)
   \left(\lambda_2 ( \lambda_2 -\lambda_3)^2+a_2^2K_3 \left  ( (\lambda_2-\lambda_3)^2 +K_3 \lambda_3 \right) \right )<0.
 \end{equation*}
Hence, Lemma \ref{l:AJ} implies that $\pm v_2$ are saddle points on $\mathcal{E}_\eta$
for all values of $a_1, a_2\in \R$.

\paragraph{Linearization at $\pm v_3\in V_3$.} Finally, for the column vector $v_3=(0,0,1)$ we have,
\begin{equation*}
\beta=(\lambda_1- \lambda_3) (\lambda_2-\lambda_3)\lambda_3>0, \qquad
   \alpha= -a_1a_2 K_3(\lambda_1- \lambda_2)\lambda_3.
 \end{equation*}
 By  Lemma \ref{l:AJ}, we conclude that $\pm v_3$ are a source and a sink on the ellipsoid $\mathcal{E}_\eta$
if $a_1\neq 0$ and $a_2\neq 0$, and instead two linear centers if either $a_1$ or $a_2$ vanish.

\begin{remark}
\label{rmk:centers}
 It can be shown that the linear centers $\pm v_1$ on the plane $W_1$ occurring when $a_2=0$, and $\pm v_3$ on the plane $W_3$ occurring when $a_1a_2=0$, are actually   {\em nonlinear centers} on the corresponding constant energy ellipsoids. For instance, 
let us consider  the equilibria $\pm v_3$ in the case $a_2=0$.  On a neighborhood of  $\pm v_3$ on the energy ellipsoid $\mathcal{E}_\eta$ ($\eta=E(v_3)$) we may write 
$$\Omega_3 =\pm  \frac{1}{\sqrt{\lambda_3}}\sqrt{ 2\eta -\lambda_2\Omega_2^2 -(\lambda_1+a_1^2K_3)\Omega_1^2},$$
with the  equilibrium point $\pm v_3$ corresponding to $(\Omega_1,\Omega_2)=(0,0)$.  Then the
 system \eqref{eq:motionB} restricted to $\mathcal{E}_\eta$ is locally equivalent to a planar system of the type
\begin{equation}
\label{eq:planarsystem}
\dot{\Omega}_1 =P( \Omega_1  , \Omega_2 ), \qquad \dot{\Omega}_2 =Q( \Omega_1  , \Omega_2 ),
\end{equation} 
having an isolated equilibrium point at the origin which, by our previous analysis, is a linear center. As may be verified,
this planar system possesses the symmetry 
$$P( \Omega_1  ,- \Omega_2 )= -P( \Omega_1  , \Omega_2 ) , \qquad Q( \Omega_1  , -\Omega_2 )=Q( \Omega_1  , \Omega_2 ).$$ 
Therefore,  a theorem by Poincar\'e (see e.g. \cite[Theorem 4.6571, p.122]{NemStepanov}) implies that $(0,0)$ is a nonlinear center. 
The above planar symmetry is inherited from the {\em reversibility} of the full 3D system 
with respect to the involution  
\begin{equation*}
\Sigma^{(2)}:\R^3\to \R^3, \qquad \Sigma^{(2)}(\Omega_1,\Omega_2,\Omega_3)=(\Omega_1,-\Omega_2,\Omega_3),
\end{equation*}
 when $a_2=0$ (i.e. if $t\mapsto \Omega(t)$ is a solution of \eqref{eq:motionB}, so is $t\mapsto \Sigma^{(2)}\Omega(-t)$).
A similar  approach may be used to analyze the equilibria  $\pm v_3$ when $a_2 \neq 0$, 
$a_1=0$. This time the conclusion follows since the resulting planar system \eqref{eq:planarsystem} possesses
the symmetry
$$P(- \Omega_1  , \Omega_2 )= P( \Omega_1  , \Omega_2 ) , \qquad Q( -\Omega_1  , \Omega_2 )=-Q( \Omega_1  , \Omega_2 ),$$ 
which is inherited from the reversibility of   \eqref{eq:motionB} 
with respect to 
\begin{equation*}
\Sigma^{(1)}:\R^3\to \R^3, \qquad \Sigma^{(1)}(\Omega_1,\Omega_2,\Omega_3)=(-\Omega_1,\Omega_2,\Omega_3).
\end{equation*}
The proof that $\pm v_1$ are nonlinear centers on the constant energy ellipsoids when $a_2=0$ is analogous and ultimately follows
from the reversibility of  \eqref{eq:motionB}  with respect to $\Sigma^{(2)}$. 
\end{remark} 

\subsection{Existence of an invariant measure}
\label{ss:measureExample}

We are now ready to discuss existence of an invariant measure which is the  point that we wish to illustrate with our example.

We start by  rewriting the reduced equations of motion \eqref{eq:motionB} as $\dot \Omega = X(\Omega)$ where the vector field $X:\R^3\to \R^3$ is given by
\begin{equation*}
\label{eq:motionX}
X(\Omega)= \mathbb{K}_a^{-1} \left  (  ( \mathbb{B}_a\Omega)  \times \Omega \right ). 
\end{equation*}
 Given that $X$ is homogeneous quadratic in  $\Omega$, it follows from Proposition 1 in Kozlov's paper \cite{Kozlov88} 
that a volume form $f(\Omega) \, d\Omega_1d\Omega_2d\Omega_3$ with  {\em strictly positive} and $C^1$ density $f:\R^3\to \R$ is invariant if and only if
 the standard euclidean 
volume form $d\Omega_1d\Omega_2d\Omega_3$ is also invariant and $f$ is a first integral. Therefore, necessary and sufficient
 conditions of existence of an  invariant 
measure  with a strictly positive $C^1$ density may   be obtained by requiring that  $\mbox{div} (X)$  vanishes for all values of $\Omega\in \R^3$.
A calculation, which is conveniently performed with the help of a symbolic algebra software, yields
\begin{equation*}
\mbox{div} (X)(\Omega)=\frac{\lambda_3K_3}{\det(\mathbb{K}_a)}\left ( -a_2\lambda_1 \Omega_1 +a_1\lambda_2\Omega_2 +a_1a_2(\lambda_1-\lambda_2)\Omega_3 \right ).
\end{equation*}
Considering that  $\lambda_j>0$, $j=1,2,3$, the above expression vanishes for all  $\Omega\in \R^3$ if and only if $a_1=a_2=0$.
Therefore, we have the following.

\begin{proposition}
\label{prop:kozlov}
The reduced equations of motion \eqref{eq:motionB}  possess an invariant volume form \newline $f(\Omega) \, d\Omega_1d\Omega_2d\Omega_3$ with a strictly positive 
density function $f:\R^3\to \R$ of class $C^1$ if and only if the vector of forbidden rotations is parallel to the axis $E_3$ of the rotor, i.e. if $a_1=a_2=0$.
\end{proposition}

\begin{remark}
The above proposition may also be proved applying the criterium given in Theorem 2 in \cite{Kozlov88} (or any of its reformulations or generalizations e.g. \cite{Jov98,ZeBloch03,FedGNMa2015}). 
\end{remark}

\begin{remark}
\label{rmk:Euler}
If  $a_1=a_2=0$,  then $\mathbb{K}_a=\mathbb{B}_a$ so the reduced equations \eqref{eq:motionB} become
 $\mathbb{K}_a\dot \Omega =(  \mathbb{K}_a  \Omega )\times \Omega$ which
are Euler's equations for a free rigid body with inertia tensor $\mathbb{K}_a$. As is well-known, the latter equations are  Hamiltonian with respect to the Lie-Poisson structure on  the coalgebra $\so(3)^*\simeq \R^3$.
Therefore, in the terminology of nonholonomic systems, we say that the system is {\em Hamiltonizable} for $a_1=a_2=0$.
\end{remark}

Our example  illustrates  that a more delicate analysis of the dynamics, aimed at understanding
 obstructions for the existence of attractors in the phase space, should necessarily 
consider the invariance of a wider class of measures. Specifically, the following generalization of Proposition \ref{prop:kozlov} holds
(see section \ref{S:measures} for the definition  of measures of class $\A$).

\begin{theorem}
\label{th:main-example}
The reduced equations of motion \eqref{eq:motionB}  possess an invariant  measure of class $\A$ on $\R^3$
 if and only if $a_2=0$. Namely, if and only if the vector of forbidden rotations lies on the plane generated
 by the largest and smallest axes of inertia of the carrier.
\end{theorem}
\begin{proof} If $a_2\neq 0$ then, by the analysis of section \ref{ss:stability}, the restricted dynamics to a constant energy ellipsoid $\mathcal{E}_\eta$, with $\eta>0$,
has a sink at one of  the  equilibrium points  $\pm v_1$.  Hence, the extension of Proposition \ref{P:attractor-measure} to manifolds
rules out the existence of an invariant  measure of class $\A$
for the restricted flow to $\mathcal{E}_\eta$ (see the comments at the end of Section \ref{S:measures}). 
The same conclusion about the unrestricted system \eqref{eq:motionB} on $\R^3$
can be obtained by applying the extension of Proposition \ref{P:attractor-measure} 
indicated in Remark \ref{rmk:measurezeroW}. Indeed,  an appropriate closed 
segment $K$ of the line $V_1$  is a compact invariant set of measure zero  which attracts a non-empty  open subset  $W\subset \R^3$. This shows that $a_2=0$ is a necessary condition for the existence of an invariant measure of class $\A$. To show that
it is sufficient, we give an explicit formula for an invariant measure of class $\A$ when $a_2=0$. Let,
\begin{equation*}
\begin{split}
&R=\left ( a_1^2K_3^2\lambda_3^2 +4(\lambda_1+a_1^2K_3)\lambda_3(\lambda_1-\lambda_2)(\lambda_2-\lambda_3)\right )^{1/2},\\
&\xi_\pm=\frac{a_1K_3\lambda_3\pm R}{2(\lambda_1-\lambda_2)(\lambda_1+a_1^2K_3)}, \qquad 
\gamma = \frac{R-a_1K_3\lambda_3}{R+a_1K_3\lambda_3},
\end{split}
\end{equation*}
and note that  $\gamma>0$ since $|a_1K_3\lambda_3|<R$. Choose $n\in \mathbb{N}$ odd  large enough so that  $M:\R^3\to \R$ defined by
\begin{equation}
\label{eq:den}
M(\Omega):= (\Omega_1-\xi_+ \Omega_3)^{n-1} \left | \Omega_1-\xi_-\Omega_3 \right |^{n\gamma-1}
\end{equation}
is of class $C^1$. A direct calculation, which is conveniently performed with a symbolic algebra software shows that, for $a_2=0$:
\begin{equation*}
\mbox{div} (M X)(\Omega)=0, \quad \forall \Omega\in \R^3.
\end{equation*}
Considering that $M$ is nonnegative and  vanishes only  along the planes
\begin{equation*}
\pi_{\pm}: \Omega_1- \xi_\pm \Omega_3=0,
\end{equation*}
which have measure zero, we conclude that the measure $M(\Omega)\, d\Omega_1d\Omega_2d\Omega_3$ is of class $\A$ and is invariant.
\end{proof}

We now make the following observations for the system in the case $a_2=0$:

\begin{enumerate}
\item[1.] The density $M$ given in  \eqref{eq:den}  only vanishes along the planes $\pi_{\pm}$ which can be shown to be invariant by the dynamics
(this actually follows from  \cite[Theorem 4]{Kozlov16}). The intersection of these planes
with a constant energy ellipsoid is comprised of the saddle points $\pm v_2$ together with four  heteroclinic connections between them.
In fact, we found the explicit form of $M$ by writing the  flow on the ellipsoid near the saddle point $v_2$ in coordinates suggested by the linearization of the stable and unstable manifolds and
proceeding in analogy with example  \eqref{eq:2Dexample}.

\item[2.] The liberty in the choice of $n$ in the  definition \eqref{eq:den} of the density $M$ comes from the fact that $$F(\Omega)=  (\Omega_1-\xi_+ \Omega_3) \left | \Omega_1-\xi_-\Omega_3 \right |^\gamma,$$
is a first integral. Therefore,  $M$ remains the density of an invariant measure when multiplied by  arbitrary  powers of $F$. 

\item[3.] The intersections of the level sets of $F$ with  the constant energy ellipsoids are generically closed curves which are periodic orbits of the dynamics. The phase space on a
fixed energy ellipsoid obtained numerically is illustrated in Figure \ref{F:a2iszero}. From the dynamical point of view, it seems undistinguishable from the phase flow of the Euler equations  
which occurs when $a_1=0$ (see Remark \ref{rmk:Euler}).  It is reasonable to expect that both systems are topologically equivalent.
\end{enumerate}

%
%

For completeness, we also provide illustrations of the phase flow on a constant energy ellipsoid obtained numerically when $a_2\neq 0$. The case $a_1\neq 0$ is shown 
in Figure \ref{F:2attractors}, where one can appreciate the  attractive  and repelling properties of the equilibria $\pm v_1$ and  $\pm v_3$ which exclude the existence of an invariant measure of class $\A$.
The case when $a_1=0$, shown in Figure \ref{F:a1iszero}, is more interesting. The equilibria $\pm v_1$ are again a source and a sink, but now $\pm v_3$ become nonlinear centers (see Remark \ref{rmk:centers}).
The numerics suggests the existence  of a homoclinic orbit emanating from $v_2$ and another from $-v_2$ which enclose the period annulus of $\pm v_3$.
It is reasonable to expect
the existence of an invariant measure whose support is not the whole phase space but the invariant
region occupied by the periodic orbits.

\begin{figure}[ht]
 \begin{subfigure}{.3\textwidth}
  \centering
  \includegraphics[width=.43\linewidth]{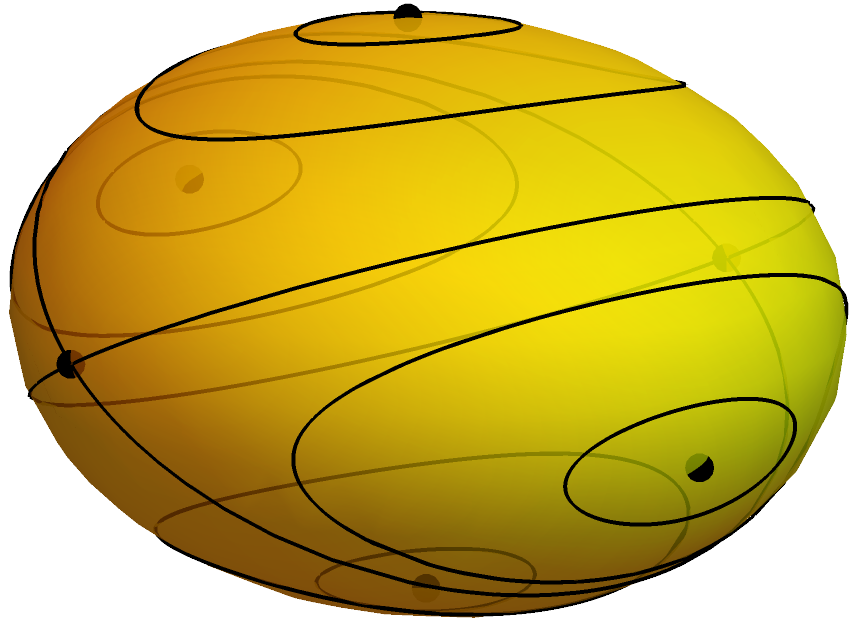}  
  \caption{$a_2= 0$.}
  \label{F:a2iszero}
\end{subfigure} 
 \begin{subfigure}{.3\textwidth}
  \centering
  \includegraphics[width=.46\linewidth]{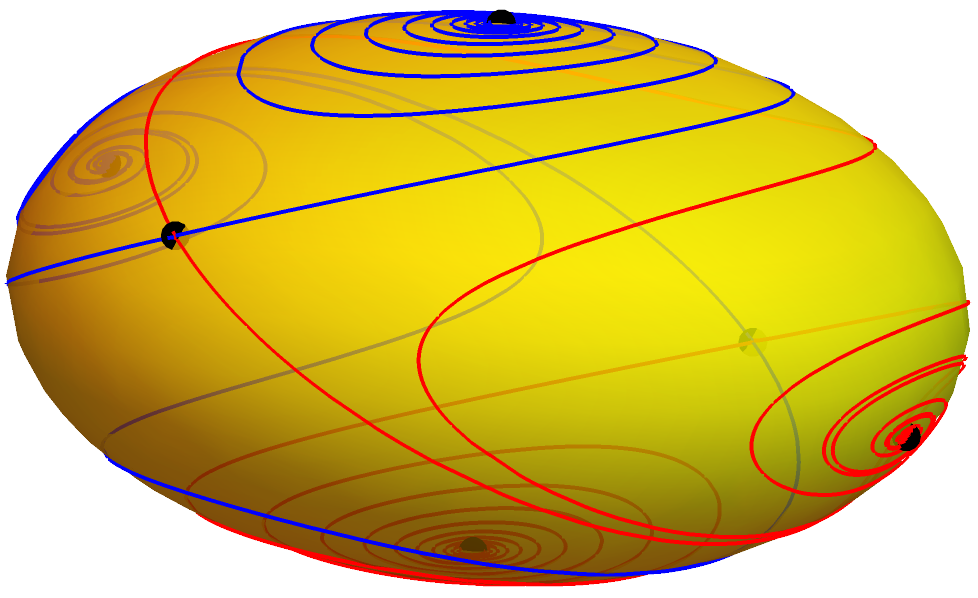}  
  \caption{$a_1\neq 0$, $a_2\neq 0$.}
  \label{F:2attractors}
\end{subfigure} 
\begin{subfigure}{.3\textwidth}
  \centering
  \includegraphics[width=.45\linewidth]{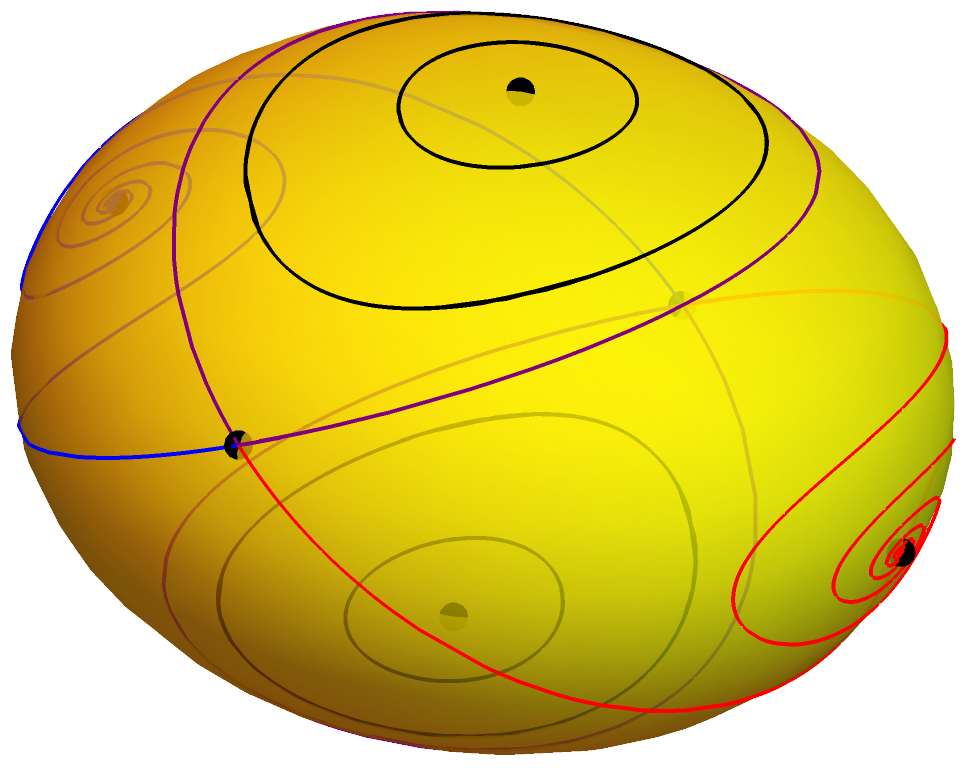}
  \caption{$a_1=0$, $a_2\neq 0$.}
  \label{F:a1iszero}
\end{subfigure} 
\caption{Phase flow on a constant energy ellipsoid for different values of $a_1$ and $a_2$. Attractors
and sources are present unless $a_2=0$.}
\label{fig:a2nonzero}
\end{figure}

\section{Final remarks}
\label{S:conclusions}

We have given a simple example of a nonholonomic system that, for certain values of the 
parameters,  possesses  an invariant measure of class $\A$ which is an obstruction to 
the existence of attractors.   Our point is that the existence of such invariant measure 
cannot be detected with the methods developed in previous references treating the problem of existence of
invariant volumes in nonholonomic mechanics
(e.g. \cite{Kozlov88,Stanchenko89,Jov98,CaCoLeMa02,ZeBloch03,FedGNMa2015}),
since they are limited to the class of measures with  strictly positive   $C^1$  densities.
Our example shows that the results of these references should be extended to a wider class of
measures if one wishes to understand obstructions or mechanisms which lead to the   existence of limit cycles (like  those
exhibited by the dynamics of the rattleback). 

We finally mention that the condition $a_2=0$, which leads to the existence of an invariant
measure of class $\A$ in the example, also leads to the reversibility of the flow with respect to the
 involution $\Sigma^{(2)}$ mentioned in Remark \ref{rmk:centers}. The relevance
of this type of discrete symmetries as obstructions to the existence attractors in the phase space of nonholonomic
systems had been indicated before (e.g. \cite[Theorem 3.3]{ZeBloch03}, \cite[Appendix]{BoMaBi13}), and may be worth
investigating further in connection with the existence of invariant measures in the class $\A$.

\section*{Acknowledgements}
LGN is grateful to the  Department of Applied Mathematics
at the University of Granada for its hospitality during his recent visit.
He also acknowledges 
support from the project MIUR-PRIN 2022FPZEES {\em Stability in Hamiltonian dynamics and beyond}. RO and AJU acknowledge funding provided by the Spanish MICINN project PID2021-128418NA-I00.

\section*{Statements and declarations} 

The authors have no competing interests to declare that are relevant to the content of this article.

\vskip 1cm

%
%

\begin{thebibliography}{99}
\let\\, \newcommand{\by}[1]{\textsc{\ignorespaces #1}\\}
  \newcommand{\title}[1]{\textsl{\ignorespaces #1}\\}
  \newcommand{\vol}[1]{{\bf{\ignorespaces #1}}}
  \newcommand{\info}[1]{\textrm{\ignorespaces #1}.}

\small  \setlength{\parskip}{0pt}
  
\bibitem{AKoch22} Arioli,~G. and Koch,~H., Some reversing orbits for a rattleback model. {\em J. Nonlinear Sci.} {\bf 32} 38 (2022).   
  
 \bibitem{Blackall}  Blackall,  C.~J.,  On volume integral invariants of non-holonomic dynamical systems.  {\em Am. J. Math.}
{\bf 63} 155--68 (1941).
  
  \bibitem{Betal2012}  Borisov, A.~V., Jalnine, A.~Y., Kuznetsov, S.~P., Sataev, I.~R. and Sedova J.~V.,  Dynamical phenomena occurring
due to phase volume compression in nonholonomic model of the rattleback. {\em Regul. Chaot. Dyn.} {\bf 17}
512--532 (2012).

  
\bibitem{BKM06}  Borisov, A.~V., Kilin, A.~A. and Mamaev, I.~S., New effects in dynamics of rattlebacks. {\em Dokl. Phys.} {\bf 51} 272--275 (2006).
  
  
\bibitem{BoMaBi13}  Borisov, A.~V.,  Mamaev,~I.S. and Bizyaev,~I.~A.,  The hierarchy of dynamics of a rigid body rolling without slipping and spinning on a plane and a sphere.
{\em  Regul. Chaotic Dyn.} {\bf 18} 277--328 (2013).
  
  
\bibitem{CaCoLeMa02}  Cantrijn F., Cort\'es J., de Le\'on M. and Mart\'in de Diego D.,  On the geometry of generalized 
Chaplygin systems. {\em Math. Proc. Camb. Phil. Soc.} {\bf 132} 323--51 (2002).
  
  
\bibitem{ClarkBloch23}  Clark, W. and  Bloch, A. M.,
Existence of invariant volumes in nonholonomic systems subject to nonlinear constraints.
{\em J. Geom. Mech.} {\bf 15} 256--286 (2023). 
  
  
 \bibitem{EvansGariepy} Evans, L.~C. and Gariepy, R.~F., {\em Measure theory and fine properties of functions.} Studies in Advanced Mathematics. CRC Press, Boca Raton, FL, (1992).
  
  \bibitem{FedGNMa2015} Fedorov Y.~N.,   Garc\'ia-Naranjo L.~C. and Marrero J.~C., Unimodularity and preservation of volumes in nonholonomic mechanics.
 {\em J. Nonlinear Sci.} {\bf 25}  203--246 (2015).
  
\bibitem{FedKoz}  Fedorov Y.~N.  and  Kozlov, V.~V.,  Various aspects of $n$-dimensional rigid body dynamics. {\em Trans. Am. Math.
Soc. Ser. 2} {\bf 168} 141--71(1995).
  
\bibitem{FedZen} Fedorov Y.~N. and Zenkov  D.~V.,  Discrete nonholonomic LL systems on Lie groups.
{\em Nonlinearity} {\bf 18} 2211--2241 (2005).

  
  



\bibitem{Jov98} Jovanovic, B., Nonholonomic geodesic flows on Lie groups and the integrable Suslov problem on SO(4). {\em J. Phys. A: Math. Gen.}  {\bf 31} 1415 (1998).


\bibitem{Jov01} Jovanovic, B., Geometry and integrability of Euler-Poincaré-Suslov equations.
{\em Nonlinearity} {\bf 14} 1555--1567  (2001).


\bibitem{Kozlov87}  Kozlov, V.~V., On the existence of an integral invariant of a smooth dynamic system. 
{\em Prikl. Mat. Mekh.} {\bf 51} 538--545 (1987), 
{\em  J. Appl. Math. Mech.} {\bf  51} 420--426 (1987).





\bibitem{Kozlov88}  Kozlov, V.~V., Invariant measures of the Euler-Poincar\'e equations on Lie algebras. {\em Funkt. Anal. Prilozh.} {\bf 22},
69--70 (Russian); English trans. {\em Funct. Anal. Appl.} {\bf 22} 58--59 (1988).



\bibitem{Kozlov16}  Kozlov, V.~V.,  Invariant measures of smooth dynamical systems, generalized functions and summation methods. {\em Izv. Math.}
{\bf  80} 342 (2016).

%

\bibitem{MacPrz}  Maciejewski A.~J. and  Przybylska M., Gyrostatic Suslov problem. {\em Russ. J. Nonlinear Dyn.}, {\bf 18} 609--627 (2022).

\bibitem{NemStepanov}  Nemytskii V.~V. and  Stepanov V.~V., {\em Qualitative Theory of Differential Equations}, Dover Publications, New York (1989).


\bibitem{Stanchenko89} Stanchenko S., Nonholonomic Chaplygin systems  {\em Prikl. Mat. Mekh.} {\bf 53} 16--23
(see also English transl. in {\em  J. Appl. Math. Mech.} {\bf 53} 11--7 (1989).



\bibitem{Suslov}  Suslov G., {\em Teoreticheskaya Mekhanika (Theoretical Mechanics)}, Moskva-Leningrad: Gostekhizdat, (1951).

\bibitem{ZeBloch03} Zenkov, D.~V. and Bloch, A.~M., Invariant measures of nonholonomic flows with internal degrees of freedom.
{\em Nonlinearity} {\bf 16} 1793--1807 (2003).
\end{thebibliography}
\end{document}